\documentclass[a4paper,11pt,dvipdfmx]{article}
\usepackage[hmargin=2.5cm,vmargin=3cm]{geometry}
\makeatletter
\let\@fnsymbol\@arabic
\makeatother

\usepackage{amsmath,amssymb,amsfonts,amscd,amsthm}
\usepackage{ascmac} 
\usepackage[mathscr]{eucal}
\usepackage{enumerate}
\usepackage{indentfirst} 
\usepackage{graphicx,color}
\newtheorem{thm}{Theorem}[section]
\newtheorem{prop}[thm]{Proposition}
\newtheorem{cor}[thm]{Corollary}
\newtheorem{dfn}[thm]{Definition}

\newtheorem{ex}[thm]{Example}
\newtheorem{rmk}{Remark}[section]
\numberwithin{equation}{section}
\allowdisplaybreaks[4]

\title{REDUCTION OF POISSON MANIFOLDS WITH HAMILTONIAN LIE ALGEBROIDS}
\author{Yuji Hirota\thanks{hirota@azabu-u.ac.jp; Division of Integrated Science, Azabu University, Sagamihara, Kanagawa 252-5201, Japan.}
 \and 
Noriaki Ikeda\thanks{nikeda@se.ritsumei.ac.jp; Department of Mathematical Sciences, 
Ritsumeikan University, Kusatsu, Shiga 525-8577, Japan.}
}
\date{}
\begin{document}
\maketitle 

\begin{abstract}
A reduction theorem for Poisson manifolds with Hamiltonian Lie algebroids is presented. 
The notion of compatibility of a momentum section is introduced to the category of Hamiltonian Lie algebroids over Poisson manifolds. 
It is shown that a compatible momentum section is a Lie algebra homomorphism, and then the quotient space of the zero level set of a compatible momentum section 
proves to be a Poisson manifold. 
\end{abstract}
\section{Introduction}

Reduction in symplectic and Poisson geometry is a powerful tool to provide us with a method to construct a new symplectic or Poisson manifolds, which has its roots in classical mechanics. 
In physics, reduction is a systematic manner to eliminate extra variables on a phase space by using symmetry or conservation laws. 
A phase space with symmetry is formulated geometrically to be a symplectic manifold admitting a suitable action of a Lie group which is called a Hamiltonian action. 
Given a Hamiltonian action, there is a map from the space to the dual of its Lie algebra such that a Hamiltonian vector field of which coincides with the infinitesimal generator of the action. 
Such a map is called a momentum map. Most of the readers know that reduction and momentum map theory are inseparable from one another. 
One of the main reduction theorems is due to J. E. Marsden and A. Weinstein \cite{MWred74}--what is called the Marsden-Weinstein reduction.

There are a lot of directions of developing reduction and momentum map theories. 
As one of such directions, it is possible to extend symplectic manifolds to more general ones, 
such as Poisson manifolds, multisymplectic manifolds, hyper/quaternionic K\"{a}hler manifolds and so on. 
It goes without saying that reduction of Poisson manifolds with symmetries has been being 
studied for many years by a lot of mathematicians, originating from the works by Euler, Lagrange and so on. 
For further details and the history including symplectic cases, refer to \cite{MWsom01, ORmom04} and references therein. 
Reduction and momentum map theory in the category of multisymplectic manifolds has been actively developing involving another research fields 
such as classical field theories, higher structures. We just mention some papers as examples \cite{Bred21,BMR22,CFRZhom16}. 
As for the case of hyper/quaternionic K\"{a}hler manifolds, the theory can be found in the work by K. Galicki \& H. B. Lawson and N. Hitchin 
\cite{Ggene87, GLqua88, HKLR87} for instance. 
The momentum map for quaternionic K\"{a}hler manifolds was proposed by K. Galicki and H. B. Lawson. It seems to be different from the usual one for symplectic manifolds. 
In fact, it proves to be a compatible homotopy momentum section of a certain bundle-valued $1$-plectic manifold by the authors \cite{HIgeo24}. 

Recently, the notion of Hamiltonian Lie algebroids over (pre-)symplectic manifolds was introduced by C. Blohmann and A. Weinstein in \cite{BWham18}. 
It is some kinds of generalization of symplectic manifolds with Hamiltonian actions. Objects corresponding to momentum maps are called momentum sections, which are defined to be 
sections of the dual of Lie algebroids with some conditions. For example, the usual momentum map $J:M\to \mathfrak{g}^*$ is thought of as a section of an action algebroid 
$\mathfrak{g}\ltimes M\to M$ in the context of Hamiltonian Lie algebroids. 
The structures of Hamiltonian algebroids can be found in some Hamiltonian mechanics and gauged nonlinear sigma models \cite{HIhom22,Imom19}. 
The authors addressed reduction of them in wider settings in \cite{HIgeo24}. 
Around the same time, C. Blohmann, S. Ronchi and A. Weinstein have extend them to the category of Poisson manifolds in \cite{BRWham23}, and have called them 
Hamiltonian Lie algebroids over Poisson manifolds. 
Incidentally, the notion of momentum sections has been generalized to the category of Courant algebroids \cite{Imom21}.

In the paper, we address reduction of Poisson manifolds with momentum sections for Hamiltonian Lie algebroids. 
We shall show that if a momentum section $\mu$ satisfies the compatibility condition (see Definition \ref{sec2:dfn_compatibility}), the zero level set 
$\mu^{-1}(\boldsymbol{0})$ is Poisson reducible in the sense of J. E. Marsden and T. Ratiu \cite{MRred86}. 
As opposed to the case of $\mathrm{Ad}^*$-equivariant momentum maps for the canonical actions of a Lie group, 
a momentum section $\mu$ for a Hamiltonian Lie algebroid $A$ over a Poisson manifold $M$ is not necessarily a Poisson map from $M$ to $A^*$ in general. 
In fact, Blohmann-Ronchi-Weinstein \cite{BRWham23} showed that $\mu$ with a certain condition preserves their brackets induced from each bivector field. 
Here, a bivector field on $A^*$ is given by the sum of the horizontal lift $\hat{\Pi}_M$ of a Poisson bivector field on $M$ 
and a Poisson bivector filed $\Pi_{A^*}$ associated to the fiber-wise linear Poisson structure on $A^*$. 
However, $\mu$ fails to be a Poisson map in general because $\hat{\Pi}_M + \Pi_{A^*}$ is not necessarily a Poisson bivector field. 
This causes difficulty of the reduction in the setting of Hamiltonian Lie algebroids. 
Our idea for the reduction of $\mu^{-1}(\boldsymbol{0})$ is to impose one geometric condition, called the compatibility condition, on a momentum section $\mu$. 
Then, $\mu$ proves to be a Poisson map, which enable us to practice a consistent reduction. 
In addition, we mention that another idea to construct a Poisson map in the Hamiltonian Lie algebroid setting has been proposed in the authors' other paper \cite{HIcom24}. 
\bigskip 

Throughout the paper, all manifolds and maps between them are assumed to be smooth. The ring of smooth functions on a smooth manifold $M$ is denoted by $C^{\infty}(M)$. 
For a vector bundle $E\to M$, the space of smooth sections of $E$ is denoted by $\varGamma(E)$. In particular, if $E=TM$, we shall write $\mathfrak{X}(M)$ for $\varGamma(TM)$. 
In the case of the exterior bundle $E=\wedge^kT^*M~(k\geqq 1)$, we denote by $\Omega^k(M)$ the space of differential $k$-forms on $M$. 
\section{Hamiltonian Lie algebroids over Poisson manifolds}
First, we make reference to the sign convention of Poisson calculus in the paper. Let $(M,\Pi)$ be a finite-dimensional Poisson manifold with a Poisson bivector field $\Pi$. 
We use the notation $\Pi^{\sharp}$ for a bundle morphism induced by $\Pi$. Concretely, $\Pi^{\sharp}$ is given by 
\begin{equation}\label{sec2:eqn_sign convention}
\Pi^{\sharp}_x:T_x^*M \longrightarrow T_xM,\quad \alpha_x\longmapsto \Pi^{\sharp}_x(\alpha_x):=(\beta_x\mapsto \langle \beta_x\wedge\alpha_x, \Pi_x\rangle)
\end{equation}
for any tangent vector $\alpha_x,\beta_x\in T_x^*M$ at each $x\in M$. Moreover, we denote the Poisson bracket of $f,\,g\in C^{\infty}(M)$ by $\{f,g\}_{\Pi}$. 
Consequently, the sign convention of Poisson calculus which we adopt in the paper is as follows:
\begin{equation*}
 \{f,\,g\}_{\Pi}=\langle\mathrm{d}f\wedge\mathrm{d}g,\,\Pi\rangle = \bigl(\Pi^{\sharp}(\mathrm{d}g)\bigr)f. 
\end{equation*}
As is well-known, the bivector field $\Pi$ defines a Lie bracket on $\Omega^1(M)$ by 
\[
\{\xi,\,\eta\}:=\mathcal{L}_{\Pi^{\sharp}\xi}\eta - \mathcal{L}_{\Pi^{\sharp}\eta}\xi + \mathrm{d}\bigl(\langle \xi\wedge \eta,\,\Pi\rangle\bigr),\quad 
\xi,\,\eta\in \Omega^1(M), 
\]
which is called the Koszul bracket. Here, $\mathcal{L}_{\Pi^{\sharp}\xi}\eta$ stands for the Lie derivative of $\eta$ by $\Pi^{\sharp}\xi$. 
We remark that 
\begin{equation}\label{sec2:eqn_Koszul}
 \mathrm{d}\{f,\,g\}_{\Pi}=-\{\mathrm{d}f,\,\mathrm{d}g\}, 
\end{equation}
or alternatively, 
\begin{equation}\label{sec2:eqn_Koszul2}
 [\Pi^{\sharp}(\mathrm{d}f),\,\Pi^{\sharp}(\mathrm{d}g)] = \Pi^{\sharp}(\{\mathrm{d}f,\,\mathrm{d}g\})
\end{equation}
hold for any $f,g\in C^{\infty}(M)$. 
\medskip 

Next, we shall recall the definition of the notion of a Hamiltonian Lie algebroid over a Poisson manifold. 
Let $A$ be a Lie algebroid over a Poisson manifold $(M,\Pi)$, whose anchor map is $\rho$. 
An $A$-differential operator is denoted by $\eth^A$. Namely, $\eth^A$ is a differential operator on the de Rham complex $\varGamma(\wedge^\bullet A^*)$, given by 
\begin{align}
 (\eth^A\theta)(\alpha_1,\cdots,\alpha_{k+1})&:= \sum_{i=1}^{k+1}(-1)^{i+1}\mathcal{L}_{\rho(\alpha_i)}\bigl(\theta(\alpha_1,\cdots,\check{\alpha}_i,\cdots,\alpha_{k+1})\bigr)\notag \\
 &\hspace{2.0em}+ \sum_{i<j}(-1)^{i+j}\theta([\alpha_i,\alpha_j],\alpha_1,\cdots,\check{\alpha}_i,\cdots,\check{\alpha}_j,\cdots,\alpha_{k+1}), \label{sec2:eqn_diff}
\end{align}
for any $\theta\in\varGamma(\wedge^kA^*),\,\alpha_1,\cdots,\alpha_{k+1}\in \varGamma(A)$. 

Suppose that $A$ is endowed with a vector bundle connection $\nabla^A$. It induces a connection on $A^*$, for which we use the same symbol $\nabla^A$. 
The induced connection on $A^*$ is defined as 
\begin{equation}\label{sec2:eqn_conn1}
(\nabla^A_X \theta)(\alpha) :=\mathcal{L}_X\bigl(\theta(\alpha)\bigr) - \theta(\nabla^A_X\alpha) 
\end{equation}
for any $X\in \mathfrak{X}(M),\,\theta\in\varGamma(A^*)$ and $\alpha\in\varGamma(A)$. 

Define a Lie algebroid connection $\mho$ on the tangent bundle $TM$ as 
\begin{equation}\label{sec2:eqn_conn2}
\mho_{\alpha}^{TM}X := [\rho(\alpha),\,X] + \rho(\nabla^A_X\alpha). 
\end{equation}
It gives rise to a Lie algebroid connection on $\wedge^pTM$ with $p\geqq 1$ (for which we denote by the same symbol as $\mho^{TM}$) by 
\begin{equation}
(\mho_{\alpha}^{TM}P)(\xi_1,\cdots,\xi_p):= \mathcal{L}_{\rho(\alpha)}\bigl(P(\xi_1,\cdots, \xi_p)\bigr) 
  - \sum_{i=1}^pP(\xi_1,\cdots,\mho^{T^*M}_{\alpha}\xi_i,\cdots,\xi_p), 
\end{equation}
where $P\in \mathfrak{X}^p(M)$ and $\xi_1,\cdots,\xi_p\in \Omega^1(M)$.
$\mho^{T^*M}_{\alpha}\xi_i~(i=1,2,\cdots, p)$ are $1$-forms defined by 
\[
 (\mho^{T^*M}_{\alpha}\xi_i)(X) := \rho(\alpha)(\xi_i(X)) - \xi_i\bigl(\mho^{TM}_{\alpha}X\bigr). 
\]

\begin{dfn}[Hamiltonian Lie algebroid over Poisson manifold \cite{BRWham23}]\label{sec2:dfn_ham Lie alg}
Let $A$ be a Lie algebroid over a Poisson manifold $(M,\Pi)$ together with a vector bundle connection $\nabla^A$. 
A momentum section $\mu$ is an element in $\varGamma(A^*)$ which satisfies 
\begin{equation}\label{sec2:dfn_ham Lie alg1}
 \rho(\alpha) = -\Pi^{\sharp}\langle \nabla^A\mu, \alpha\rangle
\end{equation}
for any $\alpha\in\varGamma(A)$. A Lie algebroid $A$ with a momentum section $\mu$ is called a Hamiltonian Lie algebroid if it satisfies both 
\begin{equation}\label{sec2:dfn_ham Lie alg2}
\mho^{TM}\Pi = 0
\end{equation}
and 
\begin{equation}\label{sec2:dfn_ham Lie alg3}
(\eth^A\mu)(\alpha_1,\alpha_2) = \Pi\bigl(\langle \nabla^A\mu, \alpha_1\rangle,\, \langle \nabla^A\mu, \alpha_2\rangle\bigr) 
\end{equation}
for any $\alpha_1,\alpha_2\in\varGamma(A)$. 
\end{dfn}
\bigskip 

The symbol $\langle \nabla^A\mu, \alpha\rangle$ in \eqref{sec2:dfn_ham Lie alg1} stands for a $1$-form given by $\langle \nabla^A\mu, \alpha\rangle(X):=(\nabla^A_X\mu)(\alpha)$ 
for $\mu\in \varGamma(A^*)$ and $\alpha\in\varGamma(A)$. For the sake of simplicity, we shall sometimes write $(\nabla^A\mu)^{\alpha}$ for $\langle \nabla^A\mu, \alpha\rangle$. Generally, 
we use the notation $\theta^{\alpha}$ for the interior product $\imath_{\alpha}\theta$ of $\theta\in \Omega^k(M,A^*)$ by a section $\alpha\in\varGamma(A)$. Namely, if 
$\theta\in \Omega^k(M,A^*)$ with $k\geqq 1$, $\theta^{\alpha}$ is a $k$-form defined by 
\[
\theta^{\alpha}(X_1,\cdots,X_k):=(\imath_{\alpha}\theta)(X_1,\cdots,X_k):= \langle \theta(X_1,\cdots,X_k),\,\alpha\rangle 
\]
where $X_1,\cdots,X_k$ are vector fields. 

\begin{rmk}
Due to the convention of \eqref{sec2:eqn_sign convention}, the sign factor in \eqref{sec2:dfn_ham Lie alg1} is slightly different from the one in \cite{BRWham23}. 
A momentum section satisfying \eqref{sec2:dfn_ham Lie alg3} is said to be bracket-compatible. 
\end{rmk}

In order to broaden the scope of momentum section, we provide the following definition: 

\begin{dfn}
Under the same setting of Definition \ref{sec2:dfn_ham Lie alg}, we let $\mathcal{S}$ be a subset of $\varGamma(A)$. 
A smooth section $\mu\in \varGamma(A^*)$ is called a momentum section of $A$ on $\mathcal{S}$ if the condition \eqref{sec2:dfn_ham Lie alg1} is satisfied for all element in $\mathcal{S}$. 
\end{dfn}
\medskip 

The fundamental example of Hamiltonian Lie algebroids might be the case where a Poisson manifold admits a Hamiltonian action of a Lie group with an $\mathrm{Ad}$-equivariant momentum map. 
In the case, the action algebroid with a trivial connection is considered as a Hamiltonian Lie algebroid (see also the subsection \ref{sec4:title}). 
Moreover, one can naturally regard a Hamiltonian Lie algebroid structure over a symplectic manifold as a Hamiltonian Lie algebroid over a Poisson manifold by a Poisson bivector 
induced from the symplectic form (see the subsection \ref{sec4:title2}). 

\begin{ex}
Let $n\geqq 1$. 
$\mathbb{R}^{2n+1}$ is a regular Poisson manifold of rank $2n$ by the bivector field $\Pi=\sum_i\frac{\partial}{\partial q_i}\wedge \frac{\partial}{\partial p_i}$ with coordinates 
$(q_1,\cdots,q_n, p_1,\cdots,p_n, y)$. Consider a cotangent Lie algebroid $A=T^*M$ over $M=\mathbb{R}^{2n+1}$ with a trivial connection $\nabla^A$ by 
$\nabla^A(f\alpha):=\mathrm{\bf d}f\otimes \alpha$ for a $1$-form in the form $f\alpha$ multiplied by a function $f\in C^{\infty}(M)$. 
Define a vector field $\mu$ by 
\[
\mu:= -\sum_{i=1}^n\biggl(q_i\frac{\partial}{\partial q_i}+p_i\frac{\partial}{\partial p_i}\biggr).
\] 
The opposite connection $\nabla^{A}$ on $TM$ is also given by a trivial connection. We have 
\[
\langle \nabla^A\mu,\,\alpha\rangle = -\sum_i(a_i\mathrm{d}q_i+ b_i\mathrm{d}p_i)
\]
for any $1$-form $\alpha=\sum_i(a_i\mathrm{d}q_i+b_i\mathrm{d}p_i)+c\,\mathrm{d}y$. 
By a simple computation, we have $\Pi^{\sharp}\langle \nabla^A\mu,\,\alpha\rangle=-\Pi^{\sharp}(\alpha)$, which means $\mu$ satisfies the condition \eqref{sec2:dfn_ham Lie alg1}. 
That is, $\mu$ is a momentum section with respect to $\nabla^A$ on $\Omega^1(M)$. 
\end{ex}
\bigskip 

Lastly, we shall introduce the notion of compatibility for momentum sections of Lie algebroids over Poisson manifolds~(c.f. \cite{HIgeo24}). 
Let $A$ be a Lie algebroid equipped with a connection $\nabla^A$ over a Poisson manifold $M$, and $\mu$ a momentum section of it. 

\begin{dfn}\label{sec2:dfn_compatibility}
$\mu$ is said to be compatible with $A$ if it satisfies the condition \eqref{sec2:dfn_ham Lie alg3} and moreover, 
$(\imath_{\alpha}\circ\nabla^A)\mu = (\mathrm{d}\circ\imath_{\alpha})\mu$ for any $\alpha\in\varGamma(A)$. 
\end{dfn}

\begin{ex}
Let $\mu$ be a momentum section for a Hamiltonian Lie algebroid $A$ over a Poisson manifold $M$ equipped with a vector bundle connection $\nabla^A$. 
Assume that the basic curvature of $\nabla^A$ is zero, that is, $\nabla^A$ satisfies the following equation 
\[
[\alpha,\nabla^A_X\beta] + [\nabla^A_X\alpha,\beta]-\nabla^A_X[\alpha,\beta] -\nabla^A_{\rho(\nabla^A_X\alpha)}\beta + \nabla^A_{\rho(\nabla^A_X\beta)}\alpha=0 
\]
for any $\alpha,\beta\in \varGamma(A)$ and $X\in \mathfrak{X}(M)$~{\rm (}see \cite{Bgeo06, HIcom24}{\rm )}. 

Let $\varGamma(A)_C$ be the subset of $\varGamma(A)$ which consists of all smooth section $\alpha$ of $A$ such that $\nabla^A\alpha=0$. 
From the assumption, we have $\nabla^A[\alpha,\beta]=0$ holds for any $\alpha,\beta\in \varGamma(A)_C$. Namely, $\varGamma(A)_C$ is closed under the Lie bracket $[\cdot,\cdot]$ in $\varGamma(A)$. 
Moreover, we have
\[
\langle \nabla^A\mu,\,\alpha\rangle = \mathrm{d}\langle\mu,\,\alpha\rangle - \langle\mu,\,\nabla^A\alpha\rangle = \mathrm{d}\langle\mu,\,\alpha\rangle. 
\]
Thus, $\mu$ is a compatible momentum section on $\varGamma(A)_C$. 
\end{ex}

If $\mu$ is compatible with $A$, then both 
\begin{equation}\label{sec2:eqn_comp01}
\rho(\alpha) = -\Pi^{\sharp}({\rm d}\mu^{\alpha}) 
\end{equation}
and 
\begin{equation}\label{sec2:eqn_comp02}
\rho(\alpha)\mu^\beta-\rho(\beta)\mu^{\alpha} - \mu^{[\alpha,\,\beta]} = \Pi(\mathrm{d}\mu^{\alpha},\,\mathrm{d}\mu^{\beta})
\end{equation}
hold for any section $\alpha,\,\beta$ of $A$. Substituting \eqref{sec2:eqn_comp01} to \eqref{sec2:eqn_comp02}, we can get the following proposition: 

\begin{prop}
If a momentum section $\mu$ is compatible with $A$, it is a Lie algebra homomorphism from $\varGamma(A)$ to $C^{\infty}(M)$, i.e., 
\[
 \mu^{[\alpha,\,\beta]} = \{\mu^{\alpha},\,\mu^{\beta}\}_{\Pi} 
\]
holds for any $\alpha,\,\beta\in \varGamma(A)$. 
\end{prop}
\section{Reduction}

Let $A$ be a Hamiltonian Lie algebroid with an anchor $\rho$ over a Poisson manifold $(M,\Pi)$. 
Suppose that a momentum section $\mu\in \varGamma(A^*)$ is compatible with $A$. 
Let $M_{\mu}$ be the preimage of the zero section of $A^*$ by $\mu$, i.e., 
\[
M_{\mu}=\bigl\{z\in M\,|\, \mu(z)=\boldsymbol{0}\in A^*|_z\bigr\}\overset{\iota}{\hookrightarrow} M. 
\] 
We assume that $M_{\mu}$ is an embedded manifold of $M$. 

The anchor map $\rho$ of $A$ gives rise to a singular distribution $\mathcal{D}_{\rho}$, which is called the characteristic distribution by 
\[
 M\ni x\longmapsto \mathcal{D}_{\rho}(x):=\mathrm{span}\bigl\{\,\rho(\alpha)_x\,|\, \alpha\in \varGamma(A)\,\bigr\}\subset T_xM. 
\]
$\mathcal{D}_{\rho}$ is integrable subbundle of $TM$. For further details of a singular distribution and the integrability, refer to \cite{DZpoi05, Sacc74, Sorb73}. 
We denote by $\mathcal{L}_{\rho}(x)$ a maximal integral manifold of $\mathcal{D}_{\rho}$ containing $x$. 
The set of $1$-forms vanishing on $\mathcal{D}_{\rho}(x)$ at each $x\in M$ is denoted by $\mathrm{ann}(\mathcal{D}_{\rho}(x))$ and is called the annihilator of $\mathcal{D}_{\rho}(x)$. 
That is, 
\[
 \mathrm{ann}\bigl(\mathcal{D}_{\rho}(x)\bigr) := \bigl\{\eta\in T_x^*M\,|\,\eta(X)=0~ \text{for any}~X\in \mathcal{D}_{\rho}(x) \bigr\},\quad x\in M.
\]
Then, we have 
\begin{prop}\label{sec3:prop_vanishing}
If $f, g$ are smooth functions on $M$ such that $(\mathrm{d}f)_z,\,(\mathrm{d}g)_z\in \mathrm{ann}\bigl(\mathcal{D}_{\rho}(z)\bigr)$ at each $z\in M_{\mu}$, then 
$(\mathrm{d}\{f,\,g\}_{\Pi})_z\in \mathrm{ann}\bigl(\mathcal{D}_{\rho}(z)\bigr)$. 
\end{prop}

\begin{proof}
Let $\alpha\in \varGamma(A)$ and $z\in M_{\mu}$. 
We remark that \eqref{sec2:eqn_comp01} holds since $\mu$ is compatible with $A$.
By the assumption, we have $\rho(\alpha)f=0$ on $M_{\mu}$. Therefore, 
\begin{align*}
\bigl(\mathrm{d}\{f,g\}_{\Pi}\bigr)_z\bigl(\rho(\alpha)_z\bigr)&= \rho(\alpha)_z\bigl(\Pi^{\sharp}(\mathrm{d}g)f\bigr) \\
&= [\rho(\alpha),\,\Pi^{\sharp}(\mathrm{d}g)]_zf + \Pi^{\sharp}(\mathrm{d}g)_z\bigl(\rho(\alpha)f\bigr) \\
&= [\rho(\alpha),\,\Pi^{\sharp}(\mathrm{d}g)]_zf \\
&= [\Pi^{\sharp}(\mathrm{d}\mu^{\alpha}),\,\Pi^{\sharp}(\mathrm{d}g)]_zf. 
\end{align*}
From \eqref{sec2:eqn_Koszul2}, \eqref{sec2:eqn_Koszul} and $\{\mu^{\alpha},\,g\}_{\Pi}=0$ it follows that  
\begin{align*}
 [\Pi^{\sharp}(\mathrm{d}\mu^{\alpha}),\,\Pi^{\sharp}(\mathrm{d}g)]_zf &= -\bigl\{f,\,\{\mu^{\alpha},g\}_{\Pi}\bigr\}_{\Pi}(z)=0. 
\end{align*}
Hence, we have $\bigl(\mathrm{d}\{f,g\}_{\Pi}\bigr)_z\bigl(\rho(\alpha)_z\bigr)=0$. Since $\alpha\in \varGamma(A)$ and $z\in M_{\mu}$ are arbitrary, we have the desired result. 
\end{proof}

We write $\mathcal{E}_{\rho}$ for the set of local flows of vector fields in $\mathcal{D}_{\rho}$, i.e., 
\[
\mathcal{E}_{\rho} = \bigl\{\,\phi_t^X\,|\,\text{$\phi_t^X$ is a local flow of $X=\rho(\alpha)$ for some $\alpha\in \varGamma(A)$}\,\bigr\} 
\]
and consider the pseudogroup $\mathcal{P}_{\rho}$ of transformations generated by $\mathcal{E}_{\rho}$  
\[
 \mathcal{P}_{\rho} = \{{\rm id}_M\} \bigcup \Bigl\{\,\phi_{t_1}^{X_1}\circ\cdots\circ\phi_{t_k}^{X_k}\, \bigm|\,k\in \mathbb{N},\, 
    \phi_{t_j}^{X_j}\in \mathcal{E}_{\rho}\, \text{or}\, (\phi_{t_j}^{X_j})^{-1}\in\mathcal{E}_{\rho}\,\Bigr\}, 
\]
where $\mathrm{id}_M$ denotes the identity map of $M$. 
\begin{rmk}
For the definition of a pseudogroup, we refer to Chapter 3 in \cite{ORmom04}. 
\end{rmk}
For the sake of simplicity, we shall use the notation $\phi_{\boldsymbol{t}}$ for an element in the form $\phi_{t_1}^{X_1}\circ\cdots\circ\phi_{t_k}^{X_k}$ in $\mathcal{P}_{\rho}$ 
and write ${\rm Dom}\,\phi_{\boldsymbol{t}}$ for the domain of $\phi_{\boldsymbol{t}}$. 
Define 
\[
 \mathcal{P}^0_{\rho} := \bigl\{\phi_{\boldsymbol{t}}\in \mathcal{P}_{\rho} \,|\, (\mu\circ\phi_{\boldsymbol{t}})(x) = \boldsymbol{0}~(x\in \mathrm{Dom}\,\phi_{\boldsymbol{t}})\bigr\}. 
\]
And moreover, for each $x\in M$ we define 
\[
\mathcal{P}^0_\rho\cdot x := \bigl\{\,\phi_{\boldsymbol{t}}(x)\,|\, \phi_{\boldsymbol{t}}\in\mathcal{P}^0_{\rho},\,{\rm Dom}\,\phi_{\boldsymbol{t}}\ni x\,\bigr\} 
\]
and call it the $\mathcal{P}^0_{\rho}$-orbit through $x$. If $z\in M_{\mu}$, the $\mathcal{P}^0_{\rho}$-orbit through $z$ coincides with the intersection of the leaf 
$\mathcal{L}_{\rho}(z)$ containing $z$ of the singular distribution $\mathcal{D}_{\rho}$ and $M_{\mu}$ i.e., $\mathcal{P}^0_\rho\cdot z = \mathcal{L}_{\rho}(z)\cap M_{\mu}$. 
The relation being in the same $\mathcal{P}^0_{\rho}$-orbit is an equivalence relation. Therefore, $M_{\mu}$ is partitioned into $\mathcal{P}^0_{\rho}$-orbits. We write $\mathcal{M}_{\rho}$ 
for the space of $\mathcal{P}^0_{\rho}$-orbits $M_{\mu}/\mathcal{P}^0_\rho$, and let $\pi_{\rho}:M_{\mu}\to \mathcal{M}_{\rho}$ be the natural quotient map. 

\begin{thm}\label{sec3:main thm}
Assume that the orbit space $\mathcal{M}_{\rho}$ is a smooth manifold such that $\pi_{\rho}$ is a submersion. 
If 
\begin{equation}\label{sec3:eqn_main}
\Pi^{\sharp}_z\bigl(\mathrm{ann}(\mathcal{D}_{\rho}(z))\bigr) \subset T_zM_{\mu} + \mathcal{D}_{\rho}(z)
\end{equation}
holds for each $z\in M_{\mu}$, then 
$\mathcal{M}_{\rho}$ is a Poisson manifold whose Poisson bracket $\{\cdot,\,\cdot\}_{\Pi/}$ is uniquely determined by 
\begin{equation}\label{sec3:eqn_main2}
 \{f,\,g\}_{\Pi/}\circ\pi_{\rho} = \{\tilde{f},\,\tilde{g}\}_{\Pi}\circ\iota 
\end{equation}
for any smooth function $f,g$ on $\mathcal{M}_{\rho}$, where $\tilde{f},\tilde{g}$ are extensions of $f\circ\pi_{\rho},\,g\circ\pi_{\rho}$, respectively, with conditions that 
$(\mathrm{d}\tilde{f})_z,\,(\mathrm{d}\tilde{g})_z\in \mathrm{ann}\bigl(\mathcal{D}_{\rho}(z)\bigr)$ at each $z\in M_{\mu}$. 

Conversely, if $\mathcal{M}_{\rho}$ is endowed with a Poisson structure by \eqref{sec3:eqn_main2}, then the condition \eqref{sec3:eqn_main} is satisfied. 
\end{thm}
\begin{proof}
Assume that \eqref{sec3:eqn_main} holds for any $z\in M_{\mu}$. Let $f,g$ be any smooth function on $\mathcal{M}_{\rho}$ 
and $\tilde{f}, \tilde{g}$ extensions of $f\circ\pi_{\rho}, g\circ\pi_{\rho}$. 
Namely, $\tilde{f}$ and $\tilde{g}$ are smooth functions on $M$ satisfying $\tilde{f}|_{M_{\mu}}=f\circ\pi_{\rho}$ and $\tilde{g}|_{M_{\mu}}=g\circ\pi_{\rho}$, respectively. 
Suppose that their differentials $\mathrm{d}\tilde{f},\,\mathrm{d}\tilde{g}$ vanish on $\mathcal{D}_{\rho}|_{M_{\mu}}$, and 
define a function $\{f,g\}_{\Pi/}$ on $\mathcal{M}_{\rho}$ as \eqref{sec3:eqn_main2}. 
Then, from Proposition \ref{sec3:prop_vanishing}, it follows that $\mathrm{d}\{\tilde{f},\,\tilde{g}\}_{\Pi}$ also vanishes on $\mathcal{D}_{\rho}|_{M_{\mu}}$. 
This implies that $\{\tilde{f},\,\tilde{g}\}_{\Pi}$ is constant on each leaf of $\mathcal{D}_{\rho}|_{M_{\mu}}$. Therefore, $\{f,g\}_{\Pi/}$ is well-defined as a function on $\mathcal{M}_{\rho}$. 
Next, we shall show that the function $\{f,g\}_{\Pi/}$ does not depend on the choice of their extensions. 
Let $\tilde{g}'$ be another extension of $g\circ\pi$ satisfying $(\mathrm{d}\tilde{g}')_z\in \mathrm{ann}\bigl(\mathcal{D}_{\rho}(z)\bigr)$ at each $z\in M_{\mu}$. 
Clearly, $\tilde{g}-\tilde{g}'=0$ on $M_{\mu}$, and $\mathrm{d}(\tilde{g}-\tilde{g}')$ vanishes on $\mathcal{D}_{\rho}|_{M_{\mu}}$. 
Therefore, using the assumption \eqref{sec3:eqn_main}, we have 
\[
\forall z\in M_{\mu};\ \bigl\langle (\mathrm{d}\tilde{f})_z,\, \Pi^{\sharp}_z\bigl(\mathrm{d}(\tilde{g}-\tilde{g}')_z\bigr)\bigr\rangle = 0, 
\]
alternatively 
\[
\forall z\in M_{\mu};\ \{\tilde{f},\,\tilde{g}\}_{\Pi}(z) = \{\tilde{f},\,\tilde{g}'\}_{\Pi}(z). 
\]
This means that the function $\{f,\,g\}_{\Pi/}\circ\pi_{\rho}$ is independent of the choice of extension. From the fact that $\{\cdot,\,\cdot\}_{\Pi}$ is a Poisson bracket, 
we see that $\mathcal{M}_{\rho}$ inherits a Poisson structure $\{\cdot,\,\cdot\}_{\Pi/}$ which satisfies $\pi_{\rho}^*\{\cdot,\,\cdot\}_{\Pi/}=\imath^*\{\cdot,\,\cdot\}_{\Pi}$. 

Conversely, we assume that the orbit space $\mathcal{M}_{\rho}$ is endowed with a Poisson structure determined by \eqref{sec3:eqn_main2}. 
We shall show the relation \eqref{sec3:eqn_main} or, equivalently, 
\begin{equation}\label{sec3:eqn_proof}
\mathrm{ann}\bigl(T_zM_{\mu}\bigr) \cap \mathrm{ann}\bigl(\mathcal{D}_{\rho}(z)\bigr)\subset \mathrm{ann}\Bigl(\Pi_z^{\sharp}\bigl(\mathrm{ann}\bigl(\mathcal{D}_{\rho}(z)\bigr)\bigr)\Bigr) 
\end{equation}
for any point $z$ in $M_{\mu}$. Fix any point $z$ in $M_{\mu}$. 
If $U_z$ is a slice chart for $M_{\mu}$ around $z$, then $\mathcal{D}_{\rho}$ is locally expressed as $\mathcal{D}_{\rho}=(U_z\cap M_{\mu})\times V$ with $V$ a vector subspace of 
$\mathbb{R}^k\oplus\mathbb{R}^{m-k}\cong \mathbb{R}^m$, where $m, k~(k<m)$ are dimensions of $M,\,M_{\mu}$, respectively. 

Take an element $\alpha_z\in \mathrm{ann}\bigl(T_zM_{\mu}\bigr) \cap \mathrm{ann}\bigl(\mathcal{D}_{\rho}(z)\bigr)$ to be arbitrary. 
From $\alpha_z\in \mathrm{ann}\bigl(T_zM_{\mu}\bigr)$, 
$\alpha_z$ can be written locally as $\alpha_z = (\mathrm{d}F)_z$  
with $F\in C^{\infty}(U_z)$ defined as 
\[
F(\boldsymbol{x})=a_1x_{k+1} + a_2x_{k+2}+\cdots +a_{m-k}x_{m},\quad \boldsymbol{x}\in U_z
\]
for some $\boldsymbol{a}=(a_1,\cdots,a_{m-k})\in \mathbb{R}^{m-k}$. Moreover, since $\alpha_z\in \mathrm{ann}\bigl(\mathcal{D}_{\rho}(z)\bigr)$, it turns out that $F$ satisfies 
\begin{equation}\label{sec3:eqn_proof2}
 F|_{U_z\cap M_{\mu}} = 0\quad \text{and}\quad \forall y\in U_z\cap M_{\mu};\ (\mathrm{d}F)_y|_{\mathcal{D}_{\rho}(y)} = \boldsymbol{0}. 
\end{equation}

Next, let $\beta_z\in \mathrm{ann}\bigl(\mathcal{D}_{\rho}(z)\bigr)$ and $G$ a function on $U_z$ such that 
\[
\beta_z=(\mathrm{d}G)_z\quad \text{and}\quad \forall y\in U_z\cap M_{\mu};\ 
(\mathrm{d}G)_y|_{\mathcal{D}_{\rho}(y)} = \boldsymbol{0}. 
\]
If $g$ is a function on $\mathcal{M}_{\rho}$ such that the extension of which is $G$, then 
\[
 \langle \alpha_z,\,\Pi^{\sharp}_z(\beta_z)\rangle = \{F,\,G\}_{\Pi}(z)=\{0,\,g\}_{\Pi/}\circ \pi_{\rho}(z)=0. 
\]
This means that $\alpha_z\in \mathrm{ann}\bigl(\Pi_z^{\sharp}\bigl(\mathrm{ann}\bigl(\mathcal{D}_{\rho}(z)\bigr)\bigr)\bigr)$. Therefore, \eqref{sec3:eqn_proof} holds. 
\end{proof}
\section{Examples}
\subsection{Poisson manifolds with canonical actions of Lie groups}\label{sec4:title}
Let $(M,\Pi)$ be a Poisson manifold admitting a canonical left action of a compact Lie group $G$ with Lie algebra $\mathfrak{g}$. 
Then, the induced Lie algebra action $\varrho:\mathfrak{g}\to \mathfrak{X}(M)$ satisfies $\mathcal{L}_{\xi_M}\Pi=0$ for any element $\xi$ in $\mathfrak{g}$, where $\xi_M=\varrho(\xi)$. 
One can define a Lie algebroid structure on a trivial vector bundle $M\times \mathfrak{g}$ whose anchor map is defined by $\rho(\xi)=-\xi_M$. 
It is called an action Lie algebroid (or a transformation algebroid, see \cite{DZpoi05}), and is denoted by $\mathfrak{g}\ltimes M$. 
Let $J:M\to \mathfrak{g}^*$ be an equivariant momentum map associated to the canonical $\mathfrak{g}$-action. From Corollary 2.6 in \cite{BRWham23}, 
$A=\mathfrak{g}\ltimes M$ is a Hamiltonian Lie algebroid with a trivial connection $\nabla^A=\mathrm{d}$. 
Here, we remark that sections of $A$ are confined to constant ones. Hence, $J$ proves to be compatible with $A$. 

Suppose that the zero element $\boldsymbol{0}$ in $\mathfrak{g}^*$ is a regular value of $J$ and let $M_J=J^{-1}(\boldsymbol{0})$. 
The characteristic distribution of $A$ is given by 
\[
 \mathcal{D}_{\rho}(x) = \mathrm{span}\bigl\{ \xi_M(x) \,|\, \xi\in \mathfrak{g}\bigr\}=\mathfrak{g}\cdot x,\quad x\in M. 
\]
If $v\in \Pi^{\sharp}_x\bigl(\mathrm{ann}(\mathfrak{g}\cdot x)\bigr)$, then, there exists a covector $\alpha\in T_x^*M$ such that 
$v=\Pi^{\sharp}_x(\alpha)$ and $\alpha(\xi_M(x))=0$. For any $\xi\in\mathfrak{g}$, we have 
\[
 \langle(\mathrm{d}J)_x(v),\, \xi_M(x)\rangle = (\mathrm{d}J^{\xi})_x(v) = -\alpha(\xi_M(x)) = 0,
\]
which implies that $v\in \ker(\mathrm{d}J)_x$. Hence, $\Pi^{\sharp}_x\bigl(\mathrm{ann}(\mathfrak{g}\cdot x)\bigr)\subset \ker(\mathrm{d}J)_x$ holds at each $x\in M$. 
As a result, the condition $\Pi^{\sharp}_x\bigl(\mathrm{ann}(\mathfrak{g}\cdot x)\bigr)\subset T_xM_{J} + \mathfrak{g}\cdot x$ is satisfied. 
Therefore, by Theorem \ref{sec3:eqn_main}, the quotient manifold $M_{J}/G$ inherits a Poisson structure from $M$ satisfying \eqref{sec3:eqn_main2}. 

\subsection{Hamiltonian Lie algebroids over symplectic manifolds}\label{sec4:title2}

First, we shall recall the definition of Hamiltonian Lie algebroids over symplectic manifolds. Let $(M,\omega)$ be a symplectic manifold and $A$ a Lie algebroid over $M$ 
whose anchor map is $\rho$. 
Suppose that $A$ is equipped with a vector bundle connection $\nabla^A$. 
One can get a connection (which is denoted by the same letter $\nabla^A$) on $A^*$ and an $A$-connection $\mho^{TM}$ on $TM$ in the same way as 
\eqref{sec2:eqn_conn1} and \eqref{sec2:eqn_conn2}, respectively. 
Moreover, $\mho^{TM}$ can be extended to an $A$-connection on $\wedge^kT^*M$ as
\begin{equation*}
(\mho_{\alpha}^{TM}\eta)(X_1,\cdots,X_k):= \mathcal{L}_{\rho(\alpha)}\bigl(\eta(X_1,\cdots, X_k)\bigr) 
  - \sum_{i=1}^k\eta(X_1,\cdots,\mho^{TM}_{\alpha}X_i,\cdots,X_k) 
\end{equation*}
for any vector fields $X_1,X_2,\cdots,X_k$, where $\eta\in \Omega^k(M)$. 

\begin{dfn}[Hamiltonian Lie algebroid over symplectic manifold \cite{BWham18}]
An element $\mu\in \varGamma(A^*)$ is called a {\rm (}$\nabla^A$-{\rm )}momentum section if 
\begin{equation}\label{sec4:dfn_ham Lie alg1}
 \nabla^A\mu = -\imath_{\rho(\alpha)}\omega
\end{equation}
holds for any $\alpha\in\varGamma(A)$. A Lie algebroid $A$ with a momentum section $\mu$ is called a Hamiltonian Lie algebroid if it satisfies both 
\begin{equation}\label{sec4:dfn_ham Lie alg2}
\mho^{TM}\omega = 0
\end{equation}
and 
\begin{equation}\label{sec4:dfn_ham Lie alg3}
(\eth^A\mu)(\alpha_1,\alpha_2) = \omega(\rho(\alpha_1),\,\rho(\alpha_2))
\end{equation}
for any $\alpha_1,\alpha_2\in\varGamma(A)$. 
\end{dfn}

We remark that the expression of \eqref{sec4:dfn_ham Lie alg1} and \eqref{sec4:dfn_ham Lie alg3} is different from the original one in \cite{BWham18} due to the sign convention. 
If the condition \eqref{sec4:dfn_ham Lie alg2} is satisfied, $A$ is said to be symplectically anchored with respect to $\nabla^A$. 
On the other hand, $\mu$ is said to be bracket-compatible if \eqref{sec4:dfn_ham Lie alg3} is satisfied. 
Denote by $\{\cdot,\, \cdot\}_{\omega}$ a Poisson structure induced by the $2$-form $\omega$, that is, 
\[
\{f,\,g\}_{\omega}:= \omega(X_f,\,X_g),\quad f,g\in C^{\infty}(M), 
\]
where $X_f, X_g$ are Hamiltonian vector fields of $f, g$, respectively. 
A Hamiltonian Lie algebroid over a symplectic manifold $(M,\omega)$ is also the one over the Poisson manifold $(M,\{\cdot,\, \cdot\}_{\omega})$. 

The notion of compatibility of a $\nabla^A$-momentum section $\mu$ is also introduced in the same way as Definition \ref{sec2:dfn_compatibility}. 
Similarly to the Poisson case, the following argument is easily verified. 

\begin{prop}
Let $\mu$ be a $\nabla^A$-momentum section. If $\mu$ is compatible with $A$, then it is a Lie homomorphism from $\varGamma(A)$ to $C^{\infty}(M)${\rm :} 
\[
 \forall \alpha,\,\beta\in \varGamma(A)\ \quad \mu^{[\alpha,\,\beta]} = \{\mu^{\alpha},\,\mu^{\beta}\}_{\omega}. 
\]
\end{prop}
\medskip 

The condition \eqref{sec3:eqn_main} is equivalent to that 
\begin{equation}\label{sec4:eqn_main}
\forall z\in M_{\mu}\ \quad \mathcal{D}_{\rho}(z)^{\omega}\subset T_zM_{\mu} + \mathcal{D}_{\rho}(z),
\end{equation}
where $\mathcal{D}_{\rho}(z)^{\omega}$ denotes the symplectic orthogonal space of $\mathcal{D}_{\rho}(z)$ in $T_zM$. 
As a corollary of Theorem \ref{sec3:main thm}, we have 

\begin{cor}
The orbit space $\mathcal{M}_{\rho}$ is a Poisson manifold with Poisson structure $\{\cdot,\,\cdot\}_{\omega/}$ uniquely determined by  
\begin{equation*}
 \{f,\,g\}_{\omega/}\circ\pi_{\rho} = \{\tilde{f},\,\tilde{g}\}_{\omega}\circ\iota 
\end{equation*}
for any smooth function $f,g$ on $\mathcal{M}_{\rho}$ if and only if the condition \eqref{sec4:eqn_main} is satisfied. 
\end{cor}

\section{Conclusion}
In the study, we have discussed reduction of Poisson manifolds with Hamiltonian Lie algebroids. 
Unlike the case of a momentum map in symplectic/Poisson geometry, momentum sections for Hamiltonian Lie algebroids are not Poisson maps in general. 
However, we have found that a momentum section is Poisson map if it satisfies the compatibility condition. 
Consequently, in Theorem \ref{sec3:main thm} we have presented the condition 
for the zero space of a momentum section for a Hamiltonian Lie algebroid structure over a Poisson manifold to be Poisson reducible. 
The result demonstrates that Poisson manifolds with momentum sections can be reducible even if they do not admit a Hamiltonian Lie group action. 

\subsection*{Acknowledgments}
We would like to thank the referees for useful comments. 
This work was supported by the research promotion program for acquiring
grants in-aid for JSPS KAKENHI Grant Number 22K03323.

\end{document}